\documentclass[12pt,reqno]{amsart}
\usepackage{amssymb, latexsym}
\usepackage{url}
\usepackage{a4wide}

\begin{document}
 \bibliographystyle{plain}

 \newtheorem{theorem}{Theorem}
 \newtheorem{lemma}{Lemma}
 \newtheorem{corollary}{Corollary}
 \newtheorem{problem}{Problem}
 \newtheorem{conjecture}{Conjecture}
 \newtheorem{definition}{Definition}
 \newtheorem*{MTP}{Mass Transference Principle}
 \newtheorem*{DS}{The Duffin-Schaeffer Theorem}
 \newtheorem*{conj}{Conjecture}
 \newcommand{\mc}{\mathcal}
 \newcommand{\rar}{\rightarrow}
 \newcommand{\Rar}{\Rightarrow}
 \newcommand{\lar}{\leftarrow}
 \newcommand{\lrar}{\leftrightarrow}
 \newcommand{\Lrar}{\Leftrightarrow}
 \newcommand{\A}{\mc{A}}
 \newcommand{\E}{\mc{E}}
 \newcommand{\HH}{\mc{H}}
 \newcommand{\C}{\mathbb{C}}
 \newcommand{\R}{\mathbb{R}}
 \newcommand{\N}{\mathbb{N}}
 \newcommand{\Q}{\mathbb{Q}}
 \newcommand{\Z}{\mathbb{Z}}
 \newcommand{\I}{\mathbb{I}}

\parskip=0.5ex

\title{Metric considerations concerning the mixed Littlewood Conjecture}

\author{Yann Bugeaud, Alan Haynes and  Sanju Velani}

\begin{abstract}
The main goal of this note is to develop a metrical theory of
Diophantine approximation within the framework  of the  de
Mathan-Teuli\'{e} Conjecture -- also known as the `Mixed Littlewood
Conjecture'. Let $p$ be a prime.
A consequence of our main result is that, for almost every real
number $\alpha$
 $$ \liminf_{n\rar\infty}n (\log n)^2 |n|_p  \, \|n\alpha\|=0   \ . $$

\vspace*{3ex}


\end{abstract}

\thanks{YB:~Research supported  by a Visiting Fellowship at the University of York.}

\thanks{AH:~Research supported by EPSRC grant EP/F027028/1.}

\thanks{SV:~Research supported by EPSRC grants EP/E061613/1 and EP/F027028/1.}

\maketitle

\section{Introduction}

The famous Littlewood Conjecture in the theory of simultaneous
Diophantine approximation  dates back to the 1930's and  asserts
that for every pair $(\alpha,\beta)$ of real numbers, we have that
\begin{equation} \label{little}
\liminf_{n\rar\infty}n\|n\alpha\|\|n\beta\|=0.
\end{equation}
Here and throughout, $\| \, . \, \| $      denotes the distance to
the nearest integer.    For background and recent `progress'
concerning  this fundamental  problem  see \cite{ekl,pvl}. However,
it is appropriate to  highlight the result of Einsiedler, Katok $\&$
Lindenstrauss  that states that the set of pairs $(\alpha,\beta)$
for which (\ref{little}) is not satisfied  is of zero Hausdorff
dimension; i.e. any exceptional set to the Littlewood Conjecture has
to be  of zero dimension.

 In 1962, Gallagher established a result which implies that if $\psi:\N\rar\R$ is a non-negative decreasing function, then for almost every  $(\alpha,\beta)$ the inequality
\begin{equation*}\label{metricLittlewood1}
\|n\alpha\|\|n\beta\|\le \psi (n)
\end{equation*}
has infinitely  (resp. finitely) many  solutions $n \in \N $   if
$\sum_{n\in\N}\psi (n)\log n $ diverges (resp. converges). In
particular, it follows that
\begin{equation} \label{littlebetter}
\liminf_{n\rar\infty}n \, (\log n)^2 \|n\alpha\|\|n\beta\|=0
\end{equation}
for almost every  pair $(\alpha,\beta)$ of real numbers. Thus from a
purely metrical point of view, Gallagher's result enables us to
`beat' Littlewood's assertion (\ref{little}) by a logarithm
squared.

The main goal of this note is to obtain a Gallagher type theorem
within the framework  of the recent de Mathan-Teuli\'{e} Conjecture
\cite{MathanTeulie} -- also known as the `Mixed Littlewood
Conjecture'.  In the following $p$ is a prime number and $| \, . \,
|_p$ is the usual $p$-adic norm.    B. de Mathan  and  O. Teuli\'{e}
conjectured that for every real number $\alpha$, we have that
\begin{equation}\label{mixedLitconj1}
\liminf_{n\rar\infty}n|n|_p\|n\alpha\|=0  \ .
\end{equation}
Various partial results exist  -- see
\cite{BugeaudDrmotaMathan,EinsiedlerKleinbock} and references
within.  Indeed, Einsiedler $\&$ Kleinbock  have shown that any
exceptional set to the  de Mathan-Teuli\'{e} Conjecture has to be of
zero  dimension.   Furthermore, let $p_1,\ldots , p_k$ be distinct
prime numbers. They also deduce, via a theorem of Furstenberg, that
if $k \geq 2$ then  for every real number $\alpha$
\begin{equation}\label{fursten}
\liminf_{n\rar\infty}  n |n|_{p_1}\cdots |n|_{p_k}\|n\alpha\|  =0  \
.
\end{equation}
This statement can be strengthened from a metrical point of view.
A consequence of our Gallagher type  theorem is that for $k \geq 1$,
\begin{equation}\label{furstenalmost}
\liminf_{n\rar\infty}  n \, (\log n)^{k+1} |n|_{p_1}\cdots
|n|_{p_k}\|n\alpha\|  =0  \
\end{equation}
for almost every  real number $\alpha$. Thus, just as with
Littlewood's conjecture, the metric statement  `beats' the de
Mathan--Teuli\'{e} assertion (\ref{mixedLitconj1}) by a logarithm
squared.

\begin{theorem}
\label{coz} Let $p_1,\ldots , p_k$ be  distinct prime numbers and
let $\psi:\N\rar\R$ be a non-negative decreasing function. Then, for
almost every real number $\alpha$ the inequality
\begin{equation*}
|n|_{p_1}\cdots |n|_{p_k}\|n\alpha\|\le \psi (n)
\end{equation*}
has infinitely (resp. finitely) many solutions $n \in \N$  if
\begin{equation*}
\sum_{n\in\N}(\log n)^k\psi (n)
\end{equation*}
diverges (resp. converges).
\end{theorem}

 This Gallagher type theorem  will be deduced as a consequence of our main result.

\begin{theorem}\label{monthm1}
Let $p_1,\ldots , p_k$ be  distinct prime numbers and  $ f_1,\ldots
,f_k:\R\rar\R$ be positive functions.  Furthermore, let
$\psi:\N\rar\R$ be a non-negative decreasing function. Then, for
almost every real number $\alpha$ the inequality
\begin{equation}\label{monsols1}
f_1(|n|_{p_1})\cdots f_k(|n|_{p_k})\|n\alpha\|\le \psi (n)
\end{equation}
has infinitely (resp. finitely) many solutions $n \in \N$ if
\begin{equation}\label{divcond2}
\sum_{n\in\N}\frac{\psi (n)}{f_1(|n|_{p_1})\cdots f_k(|n|_{p_k})}
\end{equation}
diverges (resp. converges).
\end{theorem}

\vspace*{3ex}

The final section of  the paper is devoted to  discussing  various
related metrical results and open problems.

\section{Theorem \ref{monthm1}  $\Rightarrow$    Theorem \ref{coz}}\label{proofcoz}

With reference to Theorem \ref{monthm1}, let each of the functions $
f_1,\ldots ,f_k $ be the identity function. Then, Theorem  \ref{coz}
trivially follows from Theorem \ref{monthm1}  if we can show that
$$
\sum_{n\in\N}(\log n)^k\psi (n) = \infty   \quad
\Longleftrightarrow \quad \sum_{n\in\N}\frac{\psi (n)}{ |n|_{p_1}
\cdots  |n|_{p_k}} = \infty \, .
$$
Actually we will prove the following more general lemma, which will
also be needed in Section \ref{HT}.
\begin{lemma}\label{star}
Suppose that $s\in [0,1]$ and that $p_1,\ldots ,p_k$ are distinct
primes. Furthermore, suppose that  $\psi:\N\rar\R$ is a non-negative
decreasing function.   If $s<1$ then the sum
\begin{equation}\label{divcond5}
\sum_{n\in\N}n\left(\frac{\psi(n)}{n|n|_{p_1}\cdots
|n|_{p_k}}\right)^s
\end{equation}
diverges if and only if
\begin{equation*}
\sum_{n\in\N}n^{1-s}\psi(n)^s
\end{equation*}
diverges. If $s=1$, then the divergence of (\ref{divcond5}) is
equivalent to that of
\begin{equation*}
\sum_{n\in\N}(\log n)^k\psi(n).
\end{equation*}
\end{lemma}
\begin{proof}
For the duration of the proof let us write $N:=p_1\cdots p_k$. For
one direction of the proof we begin by using the monotonicity of
$\psi$ to deduce that
\begin{align}
\sum_{n\in\N}n\left(\frac{\psi(n)}{n|n|_{p_1}\cdots |n|_{p_k}}\right)^s&=\sum_{a_1,\ldots ,a_k\ge 0}\sum_{\substack{m\in\N\\(m,N)=1}}m^{1-s}p_1^{a_1}\cdots p_k^{a_k}\psi(p_1^{a_1}\cdots p_k^{a_k}m)^s\label{refcomment}\\
&\ge\sum_{a_1,\ldots ,a_k\ge
0}\sum_{\substack{m\in\N\\(m,N)=1}}m^{1-s}\sum_{\ell=p_1^{a_1}\cdots
p_k^{a_k}m}^{p_1^{a_1}\cdots p_k^{a_k}(m+1)-1}\psi(\ell)^s.\nonumber
\end{align}
Upon interchanging the orders of summation it is apparent that the
latter quantity is equal to
\begin{align*}
\sum_{\ell\in\N}\psi(\ell)^s\sum_{\substack{m\in\N\\(m,N)=1}}\sum_{\substack{a_1,\ldots
,a_k\ge 0\\\ell/(m+1)<p_1^{a_1}\cdots p_k^{a_k}\le\ell/m}}m^{1-s}.
\end{align*}
Next by M\"{o}bius inversion and partial summation this becomes
\begin{align}
&\sum_{\ell\in\N}\psi(\ell)^s\sum_{d|N}\mu (d)d^{1-s}\sum_{m\in\N}\sum_{\substack{a_1,\ldots ,a_k\ge 0\\\ell/(md+1)<p_1^{a_1}\cdots p_k^{a_k}\le\ell/md}} \!\!\!\!\!\!\!  m^{1-s}\nonumber\\[2ex]
&\qquad =\sum_{\ell\in\N}\psi(\ell)^s\sum_{a_1,\ldots ,a_k\ge 0}\sum_{d|N}\mu (d)d^{1-s}\sum_{\substack{m\in\N\\\ell/(dp_1^{a_1}\cdots p_k^{a_k})-1/d<m\le \ell/(dp_1^{a_1}\cdots p_k^{a_k})}}  \!\!\!\!\!\!\!\!\!\!\!\!\!\!\!\!\!\!\!\!\!  m^{1-s}\nonumber\\[2ex]
&\qquad
=\lim_{L\rar\infty}\left(\sum_{\ell\le L}\left(\psi(\ell)^s-\psi(\ell+1)^s\right)\sum_{j=1}^\ell\sum_{a_1,\ldots
,a_k\ge 0} \sum_{d|N}\mu (d)d^{1-s}    \!\!\!\!\!\!\!\!\!\!\!\!\!\!
\sum_{\substack{m\in\N\\j/(dp_1^{a_1}\cdots p_k^{a_k})-1/d<m\le
j/(dp_1^{a_1}\cdots p_k^{a_k})}}  \!\!\!\!\!\!\!
\!\!\!\!\!\!\!\!\!\!\!\!\!\! m^{1-s}\right.\nonumber\\[1ex]
&\qquad\qquad\qquad\left.+ \ \psi(L+1)^s\sum_{j=1}^L\sum_{a_1,\ldots
,a_k\ge 0} \sum_{d|N}\mu (d)d^{1-s}    \!\!\!\!\!\!\!\!\!\!\!\!\!\!
\sum_{\substack{m\in\N\\j/(dp_1^{a_1}\cdots p_k^{a_k})-1/d<m\le
j/(dp_1^{a_1}\cdots p_k^{a_k})}}  \!\!\!\!\!\!\!
\!\!\!\!\!\!\!\!\!\!\!\!\!\! m^{1-s}\qquad\right).\label{partialsum1}
\end{align}
Now we focus on the sums
\begin{align}
\sum_{j=1}^\ell\sum_{\substack{m\in\N\\j/(dp_1^{a_1}\cdots
p_k^{a_k})-1/d<m\le j/(dp_1^{a_1}\cdots p_k^{a_k})}}
\!\!\!\!\!\!\!\!\!\!\!\!\!\! m^{1-s} \, .\label{partialsum1.4}
\end{align}
If we write each $j$ in the first sum as $j=idp_1^{a_1}\cdots p_k^{a_k}+r$ with $0\le i\le \ell/dp_1^{a_1}\cdots p_k^{a_k}$ and $0\le r<dp_1^{a_1}\cdots p_k^{a_k}$ then the sum over $m$ is either $i^{1-s}$ or $0$, depending on whether or not $r<p_1^{a_1}\cdots p_k^{a_k}$. Thus (\ref{partialsum1.4}) is equal to
\begin{align}
&p_1^{a_1}\cdots p_k^{a_k}\sum_{1\le i<\lfloor\ell/dp_1^{a_1}\cdots p_k^{a_k}\rfloor}i^{1-s}\nonumber\\
&\qquad+\min\left\{p_1^{a_1}\cdots p_k^{a_k},1+\ell-\left\lfloor\frac{\ell}{dp_1^{a_1}\cdots p_k^{a_k}}\right\rfloor dp_1^{a_1}\cdots p_k^{a_k}\right\}\cdot\left\lfloor\frac{\ell}{dp_1^{a_1}\cdots p_k^{a_k}}\right\rfloor^{1-s}.\label{partialsum1.5}
\end{align}

Here we break our analysis into two cases. If $s=1$ then
(\ref{partialsum1.5}) equals
\[\ell/d+O(p_1^{a_1}\cdots p_k^{a_k}),\]
and returning to
(\ref{partialsum1}) we find that it is
\begin{align*}
=\lim_{L\rar\infty}&\left(\sum_{\ell\le L}\left(\psi(\ell)-\psi(\ell+1)\right)\sum_{\substack{a_1,\ldots ,a_k\ge 0\\p_1^{a_1}\cdots p_k^{a_k}\le \ell}}\left(\frac{\varphi (N)\ell}{N}+O\left(\sum_{d|N}|\mu (d)|~p_1^{a_1}\cdots p_k^{a_k}\right)\right)\right.\\[2ex]
&\qquad\left.+\psi(L+1)\sum_{\substack{a_1,\ldots ,a_k\ge 0\\p_1^{a_1}\cdots p_k^{a_k}\le L}}\left(\frac{\varphi (N)L}{N}+O\left(\sum_{d|N}|\mu (d)|~p_1^{a_1}\cdots p_k^{a_k}\right)\right)\right).\nonumber
\end{align*}
Now note that (since $N:=p_1 \cdots p_k$ is fixed) the error terms in the inner sums are
\begin{align*}
\ll \sum_{\substack{a_1,\ldots ,a_k\ge 0\\p_1^{a_1}\cdots p_k^{a_k}\le \ell}}p_1^{a_1}\cdots p_k^{a_k}\ll \ell(\log\ell)^{k-1}.
\end{align*}
This inequality can easily be verified by induction on $k$ and we emphasize that the implied constant is dependent only on $N$. For the main terms we note that
\begin{align*}
\sum_{\substack{a_1,\ldots ,a_k\ge 0\\p_1^{a_1}\cdots p_k^{a_k}\le \ell}}1\gg(\log\ell)^k,
\end{align*}
and thus (\ref{partialsum1}) is
\begin{align*}
\gg\lim_{L\rar\infty}\left(\sum_{\ell\le L}\left(\psi(\ell)-\psi(\ell+1)\right)\ell(\log \ell)^k+\psi(L+1)(\log L)^k\right)\gg\sum_{n\in\N}(\log n)^k\psi (n).
\end{align*}
For the case when $0<s<1$ we have that
\begin{align*}
p_1^{a_1}\cdots p_k^{a_k}\sum_{1\le i\le \ell/(dp_1^{a_1}\cdots p_k^{a_k})}i^{1-s}=\frac{\ell^{2-s}}{(2-s)d^{2-s}\left(p_1^{a_1}\cdots
p_k^{a_k}\right)^{1-s}}+O\left(\left(\frac{\ell}{d}\right)^{1-s}(p_1^{a_1}\cdots p_k^{a_k})^s\right).
\end{align*}
It follows from this, (\ref{partialsum1.4}), and (\ref{partialsum1.5}) that
\begin{align}
&\sum_{j=1}^\ell\sum_{a_1,\ldots ,a_k\ge 0}\sum_{d|N}\mu (d)d^{1-s}\sum_{\substack{m\in\N\\j/(dp_1^{a_1}\cdots p_k^{a_k})-1/d<m\le j/(dp_1^{a_1}\cdots p_k^{a_k})}}m^{1-s}\nonumber\\[2ex]
&\qquad =\sum_{\substack{a_1,\ldots ,a_k\ge 0\\p_1^{a_1}\cdots
p_k^{a_k}\le \ell}} \ \sum_{d|N}\mu (d)  \, d^{1-s} \ \sum_{j=1}^\ell\sum_{\substack{m\in\N\\j/(dp_1^{a_1}\cdots p_k^{a_k})-1/d<m\le j/(dp_1^{a_1}\cdots p_k^{a_k})}}m^{1-s}\nonumber\\[2ex]
&\qquad =\sum_{\substack{a_1,\ldots ,a_k\ge 0\\p_1^{a_1}\cdots
p_k^{a_k}\le \ell}}\left(\frac{\varphi (N)  \, \ell^{2-s}}{(2-s)N(p_1^{a_1}\cdots
p_k^{a_k})^{1-s}}+O\left(2^k\ell^{1-s}(p_1^{a_1}\cdots p_k^{a_k})^s\right)\right).\label{partialsum2}
\end{align}
For the error term here we have the trivial upper bound
\begin{align*}
\sum_{\substack{a_1,\ldots ,a_k\ge 0\\p_1^{a_1}\cdots p_k^{a_k}\le
\ell}}\ell^{1-s}(p_1^{a_1}\cdots p_k^{a_k})^s\ll\ell(\log
\ell)^k.
\end{align*}
This shows that the quantity in (\ref{partialsum2}) is
bounded below by a positive constant (which depends on $N$ and $s$)
times $\ell^{2-s}$, at least for $\ell$ larger than some fixed bound. Returning to (\ref{partialsum1}) again we have
that it is
\begin{align*}
\gg\lim_{L\rar\infty}\left(\sum_{\ell\le L}\left(\psi(\ell)^s-\psi(\ell+1)^s\right)
\ell^{2-s}+\psi (L+1)^sL^{2-s}\right)\gg\sum_{n\in\N}n^{1-s}\psi (n)^s.
\end{align*}
This proves one direction of the lemma. For the other direction we
start from the observation that
\begin{align*}
\sum_{n\in\N}n\left(\frac{\psi(n)}{n|n|_{p_1}\cdots
|n|_{p_k}}\right)^s&\ll\sum_{a_1,\ldots ,a_k\ge 0}\sum_{\substack{m\in\N\\(m,N)=1}}m^{1-s}\frac{p_1^{a_1}\cdots p_k^{a_k}}{2}\psi(p_1^{a_1}\cdots p_k^{a_k}m)^s\\
&\ll \sum_{a_1,\ldots ,a_k\ge
0}\sum_{\substack{m\in\N\\(m,N)=1}}m^{1-s}\sum_{p_1^{a_1}\cdots p_k^{a_k}(m-1/2)<\ell \le p_1^{a_1}\cdots p_k^{a_k}m}\psi(\ell)^s.
\end{align*}
Since we are aiming for an upper bound this time we can drop the
condition $(m,N)=1$, and this makes things a little simpler than
before. Our bound then becomes
\begin{align*}
& \sum_{\ell\in\N}\psi(\ell)^s\sum_{0\le a_1,\ldots ,a_k\le \max_i(\log_{p_i}2\ell)}\sum_{\substack{m\in\N\\\ell/p_1^{a_1}\cdots p_k^{a_k}\le m<\ell/p_1^{a_1}\cdots p_k^{a_k}+1/2}}m^{1-s}\\
&\qquad\ll\sum_{\ell\in\N}\ell^{1-s}\psi(\ell)^s\sum_{0\le
a_1,\ldots ,a_k\le \max_i(\log_{p_i}2\ell)}(p_1^{a_1}\cdots
p_k^{a_k})^{s-1}.
\end{align*}
When $s=1$ this shows that
\begin{align*}
\sum_{n\in\N}\frac{\psi(n)}{|n|_{p_1}\cdots |n|_{p_k}}\ll
\sum_{n\in\N}(\log n)^k\psi (n),
\end{align*}
and when $s<1$ we have that
\begin{align*}
\sum_{n\in\N}n\left(\frac{\psi(n)}{n|n|_{p_1}\cdots
|n|_{p_k}}\right)^s\ll \sum_{n\in\N}n^{1-s}\psi (n)^s.
\end{align*}
This completes the proof of the lemma.
\end{proof}

\section{Proof of main result}\label{proofsec1}

The proof of the convergent case of Theorem \ref{monthm1}  is an
easy  application of the Borel-Cantelli Lemma from probability
theory.
Without loss of generality we can restrict our attention to real
numbers $\alpha$ lying within the unit interval $\I:=[0,1]$.   It
follows that we need to  determine the Lebesgue measure $| \; . \;
|$ of
$$
  \limsup_{n \in \N }   A _n           \qquad {\rm where}  \qquad
A_n :=   \{ \alpha \in \I :  {\eqref{monsols1} }  {\rm \ holds \, }
\}   \ .
$$
It is easily seen that  if  \eqref{divcond2} converges then so does
$\sum_{n\in \N}  | A_n | $  and the Borel-Cantelli Lemma implies
that the associated $\limsup$  set
 is of zero measure.   Note that in proving the  convergent case we do not require the function $\psi$  to be monotonic.

\medskip


The divergent case  constitutes the main substance of Theorem
\ref{monthm1} and will be established  as an application of the
Duffin-Schaeffer Theorem.

\subsection{Preliminaries}

The following is a consequence of an attempt by R. Duffin  and A.
Schaeffer to remove the monotonicity assumption from Khintchine's
fundamental `zero-one' law -- see \cite[Theorem 1]{DuffinSchaeffer}
and  \cite[Chapter 2]{HarmanMNT}  for background, proof and further
details.

\begin{DS}
 Let $\psi:\N\rar\R$ be a non-negative  function. Then, for almost every real number $\alpha$ the inequality
\begin{equation*}
|n \alpha- a   |\le \  \psi (n)    \qquad (a,n)=1
\end{equation*}
 has infinitely  many solutions  $(a,n) \in \Z  \times \N $ if
 \begin{equation}\label{ds1}
\sum_{n\in\N}  \psi (n)  = \infty  \quad { and} \quad
\limsup_{N\rar\infty}\left(\sum_{n\le N} \frac{\varphi (n)\psi
(n)}{n}\right) \left(\sum_{n\le N}  \psi (n) \right)^{-1}>0.
\end{equation}
\end{DS}

\vspace{2ex}

Here and throughout, $\varphi$ is the Euler phi function.  Just for
completeness, we mention that the famous  and  open Duffin-Schaeffer
Conjecture corresponds to the above statement with \eqref{ds1}
replaced by the single `natural' condition that $\sum  (\varphi
(n)\psi (n))/n $ diverges.

In the course of establishing the divergent case of Theorem
\ref{monthm1}, it will be useful to have the following elementary
fact at hand.
\begin{lemma}\label{psilem1}
Let  $p_1,\ldots , p_k$ be  distinct prime numbers and  $N\in\N$.
Then
\begin{equation*}
\sum_{\substack{n\le N\\p_1,\ldots , p_k\nmid n}}\frac{\varphi
(n)}{n}=\frac{6N}{\pi^2}\prod_{i=1}^k
\frac{p_i}{p_i+1}+O_k\left(\log N\right).
\end{equation*}
\end{lemma}
\begin{proof}
By well known properties of  $\varphi$ and the M\"{o}bius function
$\mu$ we have that
\begin{align}
\sum_{\substack{n\le N\\p_1,\ldots , p_k\nmid n}}\frac{\varphi (n)}{n}&=\sum_{\substack{n\le N\\p_1,\ldots , p_k\nmid n}}\sum_{d|n}\frac{\mu (d)}{d}   \  = \ \sum_{\substack{d\le N\\p_1,\ldots , p_k\nmid d}}\frac{\mu (d)}{d}\sum_{\substack{e\le N/d\\p_1, \ldots , p_k\nmid e}}1\nonumber\\[1ex]
&=\sum_{\substack{d\le N\\p_1,\ldots , p_k\nmid d}}\frac{\mu (d)}{d}\sum_{f|p_1\cdots p_k}\mu (f)\left(\frac{N}{fd}+O(1)\right)\nonumber\\[1ex]
&=N\left(\prod_{i=1}^k\frac{\varphi
(p_i)}{p_i}\right)\sum_{\substack{d\le N\\p_1,\ldots , p_k\nmid
d}}\frac{\mu (d)}{d^2}+O\left(2^k\sum_{\substack{d\le N\\p_1,\ldots
, p_k\nmid d}}\frac{|\mu (d)|}{d}\right).\label{phieqn1}
\end{align}
The Euler product formula for the Riemann zeta function gives us
that
\begin{equation*}
\sum_{\substack{d\le N\\p_1,\ldots , p_k\nmid d}}\frac{\mu
(d)}{d^2}=\zeta^{-1} (2)\prod_{i=1}^k(1-p_i^{-2})^{-1}+O(N^{-1})  \,
.
\end{equation*}
Combining this with (\ref{phieqn1}) completes the proof.
\end{proof}

\subsection{Proof of  divergent case of Theorem \ref{monthm1}  }

We will show that in the divergence case of Theorem \ref{monthm1},
the inequality
 \begin{equation}\label{nonmonsols1}
|n\alpha- a|\le \frac{\psi (n)}{ f_1(|n|_{p_1})\cdots
f_k(|n|_{p_k})}   \  \qquad (a,n)=1
\end{equation}
 has infinitely many solutions $(a,n) \in \Z  \times \N $ for almost every real number $\alpha$.  This clearly implies that (\ref{monsols1}) has infinitely many solutions for almost every $\alpha$ and thereby completes the proof of Theorem \ref{monthm1}.

 It is easy to see that the Duffin-Schaeffer Theorem will guarantee infinitely many solutions to \eqref{nonmonsols1} for almost  every $\alpha$ if, in addition to the divergence of (\ref{divcond2}), we have that
\begin{equation}\label{DSthmhyp}
\limsup_{N\rar\infty}\left(\sum_{n\le N} \frac{\varphi (n)\psi
(n)}{nf_1(|n|_{p_1})\cdots f_k(|n|_{p_k})}\right)\left(\sum_{n\le
N}\frac{\psi (n)}{f_1(|n|_{p_1})\cdots
f_k(|n|_{p_k})}\right)^{-1}>0.
\end{equation}

\noindent With this in mind, note that

\begin{align*}
&\sum_{n\le N} \frac{\varphi (n)\psi (n)}{nf_1(|n|_{p_1})\cdots f_k(|n|_{p_k})}\\[1ex]
&\qquad\qquad=\sum_{a_1=0}^{\lfloor\log_{p_1}N\rfloor}\cdots\sum_{a_k=0}^{\lfloor\log_{p_k}N\rfloor}\sum_{\substack{n\le N/(p_1^{a_1}\cdots p_k^{a_k})\\p_1, \ldots , p_k\nmid n}}\frac{\varphi (p_1^{a_1}\cdots p_k^{a_k}n)\psi (p_1^{a_1}\cdots p_k^{a_k}n)}{p_1^{a_1}\cdots p_k^{a_k}nf_1(p_1^{-a_1})\cdots f_k(p_k^{-a_k})}\\[2ex]
&\qquad\qquad=\sum_{a_1=0}^{\lfloor\log_{p_1}N\rfloor}\cdots\sum_{a_k=0}^{\lfloor\log_{p_k}N\rfloor}\left(\prod_{i=1}^k\frac{\varphi
(p_i)}{p_if_i(p_i^{-a_i})}\right)\sum_{\substack{n\le
N/(p_1^{a_1}\cdots p_k^{a_k})\\p_1, \ldots , p_k\nmid
n}}\frac{\varphi (n)\psi (p_1^{a_1}\cdots p_k^{a_k}n)}{n}.
\end{align*}
To deal with the inner sum we will use partial summation. First
write the collection of integers coprime to $p_1\cdots p_k$ in
increasing order as $n_1<n_2<\cdots$. Then for any function
$\psi':\N\rar\R$ and for any $M>1$ we have that
\begin{align*}
\sum_{i\le M}\frac{\varphi (n_i)\psi' (n_i)}{n_i}=&\sum_{i\le M}(\psi' (n_i)-\psi' (n_{i+1}))\sum_{j=1}^i\frac{\varphi (n_j)}{n_j}\\
&+\psi' (n_{M+1})\sum_{j=1}^M\frac{\varphi (n_j)}{n_j}.
\end{align*}
It is easy to check that Lemma \ref{psilem1} implies that
\begin{equation*}
\sum_{j=1}^i\frac{\varphi (n_j)}{n_j}\gg_k i
\end{equation*}
and if $\psi'$ is non-negative and monotonic then we can use this
fact in the inner sums of the partial summation to obtain
\begin{align*}
\sum_{i\le M}\frac{\varphi (n_i)\psi' (n_i)}{n_i}\gg_k \sum_{i\le
M}i(\psi' (n_i)-\psi' (n_{i+1}))+M\psi' (n_{M+1})=\sum_{i\le
M}\psi'(n_i).
\end{align*}
Since the implied constant here depends at most on $k$ when we
return to our above analysis we find that
\begin{align*}
&\sum_{n\le N} \frac{\varphi (n)\psi (n)}{nf_1(|n|_{p_1})\cdots f_k(|n|_{p_k})}\\
&\qquad\qquad\gg_k
\sum_{a_1=0}^{\lfloor\log_{p_1}N\rfloor}\cdots\sum_{a_k=0}^{\lfloor\log_{p_k}N\rfloor}\left(\prod_{i=1}^k\frac{\varphi (p_i)}{p_if_i(p_i^{-a_i})}\right)\sum_{\substack{n\le N/(p_1^{a_1}\cdots p_k^{a_k})\\p_1, \ldots , p_k\nmid n}}\psi (p_1^{a_1}\cdots p_k^{a_k}n)\\
&\qquad\qquad\gg_k\sum_{a_1=0}^{\lfloor\log_{p_1}N\rfloor}\cdots\sum_{a_k=0}^{\lfloor\log_{p_k}N\rfloor}\sum_{\substack{n\le N/(p_1^{a_1}\cdots p_k^{a_k})\\p_1, \ldots , p_k\nmid n}}\frac{\psi (p_1^{a_1}\cdots p_k^{a_k}n)}{f_1(p_1^{-a_1})\cdots f_k(p_k^{-a_k})}\\
&\qquad\qquad=\sum_{n\le N} \frac{\psi (n)}{f_1(|n|_{p_1})\cdots
f_k(|n|_{p_k})}  \ .
\end{align*}
This shows that hypothesis (\ref{DSthmhyp}) is satisfied and as
desired the conclusion of our theorem now follows from the
Duffin-Schaeffer Theorem.

\section{Related results and open problems}

\subsection{The Hausdorff theory}\label{HT}  For $ s > 0$, let $ {\mc H}^s (X) $ denote the $s$-dimensional Hausdorff measure of a set $X  \subseteq \R$ and let $\dim X$ denote its Hausdorff dimension.  The Mass Transference Principle \cite{MTP} allows us to deduce the following Hausdorff measure generalization of Theorem \ref{monthm1}.

\begin{theorem}\label{monthm1hs}
Let $p_1,\ldots , p_k$ be  distinct prime numbers and  $ f_1,\ldots
,f_k:\R\rar\R$ be positive functions.  Furthermore, let
$\psi:\N\rar\R$ be a non-negative decreasing function and let
$W(\psi, \mathbf{p}, \mathbf{f})$ denote the set of real numbers  in
the unit interval $\I :=  [0,1]$ for which inequality
(\ref{monsols1}) has infinitely many solutions. Then, for any $0<s
\leq 1$
\begin{equation*}
{\mc H}^s \big(W(\psi, \mathbf{p}, \mathbf{f}) \big)  =\left\{
\begin{array}{rl}
0 & { if} \;\;\;
\displaystyle \sum_{n\in\N}  \, n \;  \Big(\frac{\psi (n)}{  n \, f_1(|n|_{p_1})\cdots f_k(|n|_{p_k})} \Big)^s<\infty\,\\[4ex]
{\mc H}^s(\I) & { if} \;\;\; \displaystyle \sum_{n\in\N} \, n \;
\Big(\frac{\psi (n)}{ n \, f_1(|n|_{p_1})\cdots f_k(|n|_{p_k})}
\Big)^s=\infty\,
\end{array}
\right..
\end{equation*}
\end{theorem}

When  $s=1$, the measure  ${\mc H}^s$ coincides with one dimensional
Lebesgue measure  $| \; . \; |$ and the above theorem reduces to
Theorem \ref{monthm1}. In the case that each of the functions $
f_1,\ldots ,f_k $ is the identity function, let us write $W(\psi,
\mathbf{p})$ for $W(\psi, \mathbf{p}, \mathbf{f})$. The following
statement is a  consequence of Lemma \ref{star} and the fact that
${\mc H}^s(\I) = \infty $ when $ s < 1$.

\begin{theorem}
\label{cozhs} Let $p_1,\ldots , p_k$ be  distinct prime numbers and
let $\psi:\N\rar\R$ be a non-negative decreasing function. Then, for
any $0<s < 1$
\begin{equation*}
{\mc H}^s \big(W(\psi, \mathbf{p}) \big)  =\left\{
\begin{array}{rl}
0 & { if} \;\;\;
\displaystyle \sum_{n\in\N}  \, n^{1-s}  \;  \psi (n)^s<\infty\,\\[4ex]
\infty & { if} \;\;\; \displaystyle \sum_{n\in\N} \, n^{1-s} \;
\psi (n)^s=\infty\,
\end{array}
\right..
\end{equation*}
\end{theorem}

The fact that $s=1$ is excluded is important on two fronts.  The
first is trivial,  ${\mc H}^1 \big(W(\psi, \mathbf{p}) \big) \leq
{\mc H}^1(\I) = 1 $ and therefore can not possibly be infinite. The
other is more interesting. The sum in Theorem \ref{cozhs} at $s=1$
does not coincide with the sum appearing in Theorem \ref{coz} which
provides the criteria for the `size' of $W(\psi, \mathbf{p})$
expressed in terms on Lebesgue measure. Thus, it is impossible to
unify the Hausdorff and Lebesgue measure statements without
appealing to the `raw' sum in Theorem \ref{monthm1hs}.

A straightforward consequence of Theorem \ref{cozhs} is that
$$
\dim \, W(\psi, \mathbf{p})  \   =  \ \inf \{ s :
\textstyle{\sum_{n\in\N}  \, n^{1-s}  \;  \psi (n)^s<\infty }  \}
\, .
$$
In particular, let us consider the case when $\psi(n) = n^{-\tau}
\ (\tau > 0) $ and  write $W(\tau, \mathbf{p})$ for $ W(\psi,
\mathbf{p})$. Then, the following statement can be regarded as the
`mixed' analogue of the classical Jarn\'{\i}k--Besicovitch Theorem.

\begin{corollary}\label{cozhsJB}
Let $p_1,\ldots , p_k$ be  distinct prime numbers and let $\tau \geq
1$.  Then,
$$
\dim  W(\tau, \mathbf{p})  =  \frac{2}{\tau + 1}  \ .
$$
\end{corollary}

For background and further details regarding the general Hausdorff
measure theory of metric Diophantine approximation see
\cite{Beresnevich-Bernik-Dodson-Velani-Roth,BDV06} and references
within.

\subsection{Exponents of Diophantine approximation}

Motivated by the `mixed'  analogue of the classical
Jarn\'{\i}k--Besicovitch Theorem, we introduce the `mixed' analogue
of the classical notion of exact order. For the sake of clarity and
simplicity,  we restrict our attention to the case of one (fixed)
prime $p$. For a real number $\xi$, let $\tau_p(\xi)$ denote the
supremum of the real numbers $\tau$ such that the inequality
$$
|n|_p \| n \xi \|  \, \leq \, n^{-\tau}  \;
$$
has infinitely many solutions  $n \in \N$.  Recall, the \emph{exact
order} $\tau (\xi)$ of $\xi$ is defined to be  the supremum of the
real numbers $\tau$ such that the inequality
\begin{equation}\label{svtau}
\| n \xi \| \, \leq \, n^{-\tau}  \;
\end{equation}
has infinitely many solutions  $n \in \N$.   For every real number
$t \geq 1$, we know  that
\begin{equation}\label{dimeq9}
{\mc H}^{\frac{2}{t + 1}}  \big( \{\xi : \tau(\xi) = t\} \big)  =
{\mc H}^{\frac{2}{t + 1}}  \big( \{\xi : \tau_p(\xi) = t\}  \big) =
\ \infty
\end{equation}
and
\begin{equation}\label{dimeq}
\dim \, \{\xi : \tau(\xi) = t\} = \dim \, \{\xi : \tau_p(\xi) = t\}
= \frac{2}{t + 1}  \ .
\end{equation}
The dimension result for the classical exact order set  $ \{\xi :
\tau(\xi) = t\}  $  was first explicitly stated by  G\"{u}ting  --
see \cite{exactBDV} for a `modern' proof which also implies the
measure statement and references within for `exact order'
background. Following the basic principle exploited in
\cite{exactBDV}, it is easy to deduce the measure result for the
mixed exact order set $ \{\xi : \tau_p(\xi) = t\}$ from Theorem
\ref{cozhs} and the fact that
%
$$
W(\tau, p) \setminus W(\psi, p)   \, \subset \,  \{\xi : \tau_p(\xi)
=   t \}
$$
with $\tau  := t$ and  $\psi(n) := n^{-t} (\log n)^{-(t + 1 )}   \,
$.   Note that the Hausdorff measure result implies the lower bound for the Hausdorff dimension statement. The complementary upper bound is a consequence of the fact that $   \{\xi : \tau_p(\xi) =   t \}   \subset   W(t+ \epsilon, p ) $ for any $\epsilon > 0$.

\medskip

 On using the trivial
fact that $n^{-1} \le |n|_p \le 1$,   it follows that for any real
number $\xi$
\begin{equation*}\label{dimtriv}
\tau(\xi) \le \tau_p(\xi) \le \tau(\xi) + 1   \, .
\end{equation*}
Deeper still, for any  given $\delta$ in $[0, 1]$ and $t$
sufficiently large, it is possible to adapt  the procedure described
in \cite{Bu08} to construct explicit real numbers $\xi$ such that
$$
\tau(\xi) = t    \qquad  {\rm and \ }  \qquad  \tau_p(\xi) = t +
\delta  \, .
$$
Consequently, the set of values taken by the function $\tau_p -
\tau$ is precisely the whole interval $[0, 1]$.    We suspect that
the set of real numbers for which the classical and mixed  exact
order exponents differ ($\delta > 0$)  is of maximal dimension; that
is
$$
\dim \{\xi : \tau_p (\xi) > \tau (\xi) \} = 1   \, .
$$
Currently we are only able to prove that the dimension is positive.
Indeed this is a consequence of  showing that
$$
\dim \{\xi : \tau_p (\xi) = \tau (\xi) + 1 \}  > 0   \ .
$$
The proof relies  on being able to  construct a Cantor type subset
of  $\{\xi : \tau_p (\xi) = \tau (\xi) + 1 \}$  consisting  of real
numbers  all of whose best rational approximations have a
denominator a power of $p$.

We now turn our attention to the situation for which the  classical and mixed
exact order exponents are equal ($\delta = 0$) to a given value $t
\geq 1$.  Let   $ W(\tau)$  denote the set of real numbers  for which inequality  (\ref{svtau}) has infinitely many solutions  and observe that
$$
W(\tau) \setminus W(\psi, p)   \, \subset \,  \{\xi : \tau_p(\xi) = \tau(\xi) =t \}
$$
with $\tau  := t$ and  $\psi(n) := n^{-t} (\log n)^{-(t + 1 )}   \,
$. On combining this with Theorem \ref{cozhs} and the  classical fact that $ {\mc H}^{\frac{2}{t + 1}} (W(t))  = \infty $, it follows that
\begin{equation*}\label{dimeq9Q}
{\mc H}^{\frac{2}{t + 1}}  \big( \{\xi : \tau_p(\xi) = \tau(\xi) =t  \}  \big) =
\ \infty   \  .
\end{equation*}
 In turn it is easy to deduce, that for $t \geq 1 $
$$ \dim \{\xi : \tau_p(\xi) = \tau(\xi) =t \}     =
\frac{2}{t + 1}  \ . $$
Clearly, these results imply our `opening' results given by (\ref{dimeq9}) and (\ref{dimeq}).   However the opening results do not even imply that $ \{\xi : \tau_p(\xi) = \tau(\xi) =t \} $ is non-empty.


\subsection{Removing monotonicity }

The method of proof of Theorem \ref{monthm1} allows us to draw
conclusions even when the approximating function $\psi$ is
non-monotonic. For example we can prove the following result.
\begin{theorem}\label{nonmonthm1}
Let $p_1,\ldots , p_k$ be  distinct prime numbers and  $ f_1,\ldots
,f_k:\R\rar\R$ be positive functions.  Furthermore, let
$\psi:\N\rar\R$ be a non-negative  function. Then,  for almost every
real number $\alpha$  the inequality
\begin{equation}
f_1(|n|_{p_1})\cdots f_k(|n|_{p_k}) \ |n \alpha- a |\le  \psi (n)
\qquad (a,n)=1
\end{equation}
has infinitely many solutions if  there exists $\epsilon>0$ for
which
\begin{equation}\label{divcond3}
\sum_{n\in\N}   \varphi (n)   \, \left(\frac{\psi (n)}{   n \,
f_1(|n|_{p_1})\cdots f_k(|n|_{p_k})}\right)^{1 + \epsilon}=\infty.
\end{equation}
\end{theorem}
We point out that there are examples of non-monotonic $\psi$ for
which (\ref{divcond2}) diverges  but (\ref{monsols1}) has only
finitely many solutions almost everywhere. The Duffin-Schaeffer
counterexample at the end of \cite{DuffinSchaeffer} can easily be
modified to  show how this can happen.   In other words,
disallowing non-reduced solutions and  thereby  introducing the
Euler phi function in Theorem \ref{nonmonthm1} is absolutely
necessary when dealing with non-monotonic approximating functions.

Theorem \ref{nonmonthm1} is a trivial consequence  of a known result
regarding  the Duffin-Schaeffer Conjecture.  Basically, given $\psi$
we simply  apply Corollary 1 of  \cite{haypolvel} to the
approximating function
\begin{equation}  \label{standfunc}
 \Psi(n)  :=    \frac{\psi (n)}{    f_1(|n|_{p_1})\cdots f_k(|n|_{p_k})}  \ .
\end{equation}

%
%

\subsection{Simultaneous approximation}

So far we have restricted our attention to approximating a single
real number  $ \alpha \in \R$.  Clearly, it is natural to develop
the theory of `mixed'  simultaneous approximation in which one
considers points $(\alpha_1,\ldots,\alpha_m) \in \R^m$ and  the
system of inequalities
\begin{equation}\label{monsols1sv}
f_1(|n|_{p_1})\cdots f_k(|n|_{p_k}) \ |n\alpha_i- a_i | \le  \psi
(n) \qquad 1\le i\le m    \ .
 \end{equation}
The following result is a direct consequence of Gallagher's theorem
\cite{Gallagher65} in the classical theory of simultaneous
approximation. Indeed, given $\psi$ we simply  apply Gallagher's
theorem  to the approximating function $\Psi$ given by
(\ref{standfunc}).

\begin{theorem}   \label{thmsimgall}
Let $p_1,\ldots , p_k$ be  distinct prime numbers and  $ f_1,\ldots
,f_k:\R\rar\R$ be positive functions.  Furthermore, let
$\psi:\N\rar\R$ be a non-negative  function and $m \geq 2 $ be an
integer. Then, for almost every  $(\alpha_1,\ldots,\alpha_m) \in
\R^m$  the system of inequalities  given by (\ref{monsols1sv}) with
$
 (a_1,\ldots,a_m,n)=1 $
has infinitely (resp. finitely) many solutions $(a_1,\ldots,a_m,n)
\in \Z^m \times \N$ if
\begin{equation}\label{divcond2sv}
\sum_{n\in\N} \Big( \frac{\psi (n)}{f_1(|n|_{p_1})\cdots
f_k(|n|_{p_k})}\Big)^m
\end{equation}
diverges (resp. converges).
\end{theorem}

The following is an immediate corollary and generalizes Theorem
\ref{monthm1} to the simultaneous setting. Note that it is free of
any monotonicity assumption on $\psi$.

\begin{theorem} \label{thmsim}
Let $p_1,\ldots , p_k$ be  distinct prime numbers and  $ f_1,\ldots
,f_k:\R\rar\R$ be positive functions.  Furthermore, let
$\psi:\N\rar\R$ be a non-negative  function and $m \geq 2$ be an
integer. Then, for almost every   $(\alpha_1,\ldots,\alpha_m) \in
\R^m$  the system of inequalities  given by (\ref{monsols1sv}) has
infinitely (resp. finitely) many solutions $(a_1,\ldots,a_m,n) \in
\Z^m \times \N$ if the sum given by (\ref{divcond2sv}) diverges
(resp. converges).
\end{theorem}

\vspace{2ex}

Needless to say, the Mass Transference Principle enables us to place
Theorems  \ref{thmsimgall}  $\&$    \ref{thmsim}   within the
general setting of Hausdorff measures.  In particular, the
Hausdorff measure generalization of  Theorem  \ref{thmsim} extends
Theorem   \ref{monthm1hs}   to the simultaneous setting.

\begin{theorem}
Let $p_1,\ldots , p_k$ be  distinct prime numbers and  $ f_1,\ldots
,f_k:\R\rar\R$ be positive functions.  Furthermore, let
$\psi:\N\rar\R$ be a non-negative  function and let $W_m(\psi,
\mathbf{p}, \mathbf{f})$ denote the set of points  in the unit cube
$\I^m:=[0,1]^m$ for which the system of inequalities  given by
(\ref{monsols1sv}) has infinitely many solutions. Then, for $m \geq
2 $ and any  $0<s \leq m$
\begin{equation*}
{\mc H}^s \big(W_m(\psi, \mathbf{p}, \mathbf{f}) \big)  =\left\{
\begin{array}{rl}
0 & { if} \;\;\;
\displaystyle \sum_{n\in\N}  \, n^{m}  \;  \Big(\frac{\psi (n)}{n \, f_1(|n|_{p_1})\cdots f_k(|n|_{p_k})} \Big)^s<\infty\,\\[4ex]
{\mc H}^s(\I^m) & { if} \;\;\; \displaystyle \sum_{n\in\N} \, n^{m}
\;  \Big(\frac{\psi (n)}{  n \, f_1(|n|_{p_1})\cdots f_k(|n|_{p_k})}
\Big)^s=\infty\,
\end{array}
\right..
\end{equation*}
\end{theorem}

%

\vspace*{4ex}

\subsection{An intriguing `multiplicative' problem}

Let  $\psi:\N\rar\R$ be a non-negative decreasing function. It is
natural to  attempt to generalize Theorem \ref{monthm1} so as to
incorporate  approximations of the form
\begin{equation}\label{multiplereal}
f_1(|n|_{p_1})\cdots
f_k(|n|_{p_k})\|n\alpha_1\|\cdots\|n\alpha_m\|\le\psi (n)  \ ,
\end{equation}
where $(\alpha_1,\ldots ,\alpha_m)\in \R^m $.   We would expect to
be able to prove that for almost every  $(\alpha_1,\ldots
,\alpha_m)\in \R^m $  the inequality   given by (\ref{multiplereal})
has infinitely (resp. finitely) many solutions $n \in \N$ if
$$
\sum_{n\in\N}  \;  (\log n)^{m-1}  \, \frac{\psi
(n)}{f_1(|n|_{p_1})\cdots f_k(|n|_{p_k})}
$$
diverges (resp. converges).  The method which we used to prove
Theorem \ref{monthm1} would work for this more general setup if we
could establish the following `multiplicative'  generalization of
the Duffin-Schaeffer Theorem.
\begin{conj}
Let $\psi:\N\rar\R$ be a non-negative  function and let $m\in\N$.
Then, for almost every $(\alpha_1,\ldots ,\alpha_\ell)\in \R^m$ the
inequality
\begin{equation}\label{multiplereal2}
\|n\alpha_1\|\cdots\|n\alpha_m\|\le \  \psi (n)
\end{equation}
has infinitely  many solutions  $n\in\N $ if

$$
\sum_{n\in\N} (\log n)^{m-1}\psi (n) = \infty
$$
and
$$
 \limsup_{N\rar\infty}\left(\sum_{n\le N} \left(\frac{\varphi (n)}{n}\right)^m(\log n)^{m-1}\psi (n)\right)  \left(\sum_{n\le N}  (\log n)^{m-1}\psi (n) \right)^{-1} \, >  \, 0  \ .
$$
\end{conj}

\vspace*{2ex}

\noindent

We are `morally'  able to prove this conjecture. More precisely, we
are able to show that the associated  $\limsup$ set of points
satisfying (\ref{multiplereal2}) is of positive measure.    The
missing piece in our attempted proof is that we have been unable to
establish a zero-one law for this $\limsup$ set.  Indeed,
establishing such a law would be of interest in its own right.

%

\vspace*{1ex}

\noindent{\bf Problem.}  Let  $W_m^*(\psi)$ denote the set of points
in the unit cube $\I^m $ for which the inequality given by
(\ref{multiplereal2})  has infinitely  many solutions.  Prove that
the $m$-dimensional Lebesgue measure of the $\limsup$ set
$W_m^*(\psi)$ is either zero or one.

\vspace*{1ex}

\noindent We make one final comment concerning the conjecture. For
$m \geq 2 $,  the conjecture is likely to be true without imposing
the $\limsup$ condition. In other words,  the divergent sum
condition is all that is required.
%


\vspace*{2ex}

Attempting  to generalize Theorem \ref{monthm1} as above can be
viewed as developing a metrical theory of Diophantine approximation
within the framework of the following generalization of the de
Mathan-Teuli\'{e} Conjecture:
\begin{equation}\label{mlsv}
\liminf_{n\rar\infty}  n |n|_{p_1}\cdots |n|_{p_k}
\|n\alpha_1\|\cdots\|n\alpha_m\|  =  0   \qquad \forall  \
(\alpha_1,\ldots,\alpha_m) \in \R^m \ .
\end{equation}
As already mentioned in the introduction the statement is true when
$k \geq 2$ -- see also Section \ref{Bourg} below.  Thus, let us
assume that $k=1$ and note that  establishing (\ref{mlsv})  for
$m=1$ (the de Mathan-Teuli\'{e} Conjecture) trivially implies
(\ref{mlsv}) for all $m$.   However, one could argue that
establishing (\ref{mlsv}) should get easier the larger we take $m$.
Nevertheless, nothing seems to be known.

\begin{conj}
Let $p$ be a prime and $m\in\N$. For $m$ sufficiently large we have
that
$$
\liminf_{n\rar\infty}  n |n|_{p}  \|n\alpha_1\|\cdots\|n\alpha_m\|
=  0   \qquad \forall  \ (\alpha_1,\ldots,\alpha_m) \in \R^m \ .
$$
\end{conj}

\vspace*{4ex}

\subsection{Various strengthenings   }  \label{Bourg}
In view of the recent work of Bourgain, Lindenstrauss, Michel  $\&$
Venkatesh \cite{BLMV}, it is possible to strengthen (\ref{fursten})
and therefore (\ref{mlsv}) when $k \geq2$. Indeed, given distinct
primes $p_1$ and $p_2$, it follows from Theorem 1.8 of \cite{BLMV}
that there exists a small positive constant $\kappa$ such that for
every real number $\alpha$
\begin{equation}\label{furstenbis}
\liminf_{n\rar\infty}  n (\log \log \log n)^{\kappa} |n|_{p_1}
|n|_{p_2}\|n\alpha\|  =0  \ .
\end{equation}
To be precise,  Theorem 1.8 of \cite{BLMV} can be applied unless
$\alpha$ is a Liouville number.  However, for Liouville numbers  the
statement given by (\ref{furstenbis}) trivially holds.

Returning to the  original de Mathan-Teuli\'{e} Conjecture,  for
quadratic numbers  $\alpha $ the stronger statement
\begin{equation}  \label{yannquad}
\liminf_{n\rar\infty}  n (\log n) |n|_{p}\|n\alpha\|  < + \infty \
\end{equation}
has been established in \cite{MathanTeulie}. Another class of
(transcendental) real numbers with bounded partial quotients and for
which (\ref{yannquad})  holds is given in \cite[Theorem
1]{BugeaudDrmotaMathan}.  It would be highly desirable to determine
whether or not there exist $\alpha$ for which  (\ref{yannquad})  is
violated.   It is shown in \cite{BuMo} that the set of real numbers
$\alpha$ for which
$$
\liminf_{n\rar\infty}  n (\log n)^2 |n|_{p} \|n\alpha\|   > 0
$$
has full Hausdorff dimension.  Most recently, a consequence of the main result  in \cite{mixdbsv} is that the set of real numbers
$\alpha$ for which
$$
\liminf_{n\rar\infty}  n (\log n)  \, (\log\log n) \,   |n|_{p} \|n\alpha\|   > 0
$$
has full Hausdorff dimension.  In all likelihood the full dimension statement is  true without the $\log\log n$ term.

%

\vspace{4ex}

\noindent{\em Acknowledgements. }
We would like to the thank the referee for carefully reading the paper and pointing out various oversights  in the original  version. The referee also observed that Lemma \ref{star} can be proved without using M\"{o}bius inversion, by replacing the condition in (\ref{refcomment}) that $(m,N)=1$ by $m\equiv 1\mod N$, and replacing the equality by an inequality.

SV would like to thank Boris Adamczewski for independently mentioning the `starter' problem  during the Gelfond-100 conference at Moscow State in 2007.  He would  also like to thank the fab three -- Bridget, Iona and Ayesha --
for putting up with and helping `hop along' this year -- much  appreciated!

\vspace*{2ex}

\noindent Yann Bugeaud: Math\'ematiques, Universit\'e de Strasbourg,

\vspace{0mm}

\noindent\phantom{Yann Bugeaud: }7, rue Ren\'e Descartes, 67084
Strasbourg Cedex, France.


\noindent\phantom{Yann Bugeaud: }e-mail: bugeaud@math.u-strasbg.fr

\vspace{3mm}

\noindent Alan K. Haynes: Department of Mathematics, University of
York,

\vspace{0mm}

\noindent\phantom{Alan K. Haynes: }Heslington, York, YO10 5DD,
England.


\noindent\phantom{Alan K. Haynes: }e-mail: akh502@york.ac.uk

\vspace{3mm}

\noindent Sanju L. Velani: Department of Mathematics, University of
York,

\vspace{0mm}

\noindent\phantom{Sanju L. Velani: }Heslington, York, YO10 5DD,
England.


\noindent\phantom{Sanju L. Velani: }e-mail: slv3@york.ac.uk

\end{document}